\pdfoutput=1
\RequirePackage{ifpdf}
\ifpdf 
\documentclass[pdftex]{sigma}
\else
\documentclass{sigma}
\fi

\numberwithin{equation}{section}

\newtheorem{Theorem}{Theorem}[section]
\newtheorem*{Theorem*}{Theorem}
\newtheorem{Corollary}[Theorem]{Corollary}

\newtheorem{Proposition}[Theorem]{Proposition}
 { \theoremstyle{definition}

\newtheorem*{geomdef}{Geometric Definition}
\newtheorem*{standef}{Standard Definition}

\newtheorem*{xpl}{Example}
 }

\usepackage{mathrsfs}
\def\CP{\mathbb{CP}}
\def\CC{\mathbb{C}}
\def\ZZ{\mathbb{Z}}
\def\RR{\mathbb{R}}

\DeclareMathOperator{\rank}{rank}
\DeclareMathOperator{\todd}{Td}
\DeclareMathOperator{\Ext}{Ext}

\begin{document}
\allowdisplaybreaks

\newcommand{\arXivNumber}{2108.01739}

\renewcommand{\thefootnote}{}

\renewcommand{\PaperNumber}{102}

\FirstPageHeading

\ShortArticleName{Twistors, Self-Duality, and Spin$^c$ Structures}

\ArticleName{Twistors, Self-Duality, and Spin$\boldsymbol{{}^c}$ Structures\footnote{This paper is a~contribution to the Special Issue on Twistors from Geometry to Physics in honor of Roger Penrose. The~full collection is available at \href{https://www.emis.de/journals/SIGMA/Penrose.html}{https://www.emis.de/journals/SIGMA/Penrose.html}}}

\Author{Claude LEBRUN}

\AuthorNameForHeading{C.~LeBrun}

\Address{Department of Mathematics, Stony Brook University, Stony Brook, NY 11794-3651 USA}
\Email{\href{mailto:claude@math.stonybrook.edu}{claude@math.stonybrook.edu}}
\URLaddress{\url{http://www.math.stonybrook.edu/~claude/}}

\ArticleDates{Received August 02, 2021, in final form November 15, 2021; Published online November 19, 2021}

\Abstract{The fact that every compact oriented 4-manifold admits spin$^c$ structures was proved long ago by Hirzebruch and Hopf. However, the usual proof is neither direct nor transparent. This article gives a new proof using twistor spaces that is simpler and more geometric. After using these ideas to clarify various aspects of four-dimensional geometry, we then explain how related ideas can be used to understand both spin and spin$^c$ structures in any dimension.}


\Classification{53C27; 53C28; 57R15}

\renewcommand{\thefootnote}{\arabic{footnote}}
\setcounter{footnote}{0}

 \section[Twistor spaces and spin c structures]{Twistor spaces and spin$\boldsymbol{{}^c}$ structures}

 Every compact oriented 4-manifold admits spin$^c$ structures. The standard proof of this fact is due to Hirzebruch and Hopf \cite{hiho},
 although the result had previously been hinted at by Whitney \cite{whitney}. For readable modernized English-language versions
 of the Hirzebruch--Hopf proof, see Killingback and Rees~\cite{krees} or Gompf and Stipsicz \cite[Section~5.7]{gost}.

 However, the Hirzebruch--Hopf proof is so indirect that it does not really involve the notion of a spin$^c$ structure at all, and it
 proceeds by so completely isolating
 the topological issues from
 the geometric motivation as to make it seem rather formal and
 unenlightening. The main purpose of this article is to give a self-contained
 proof of this important fact that is based on ideas from twistor theory. In the process, we will also see how this result is inextricably related to other fundamental aspects of $4$-dimensional geometry. The article then concludes by putting this $4$-dimensional story in the context of a twistor approach to spin and spin$^c$ structures in other dimensions.

 Let us begin by recalling
 that dimension four is profoundly exceptional for both differential topology and differential geometry.
 This idiosyncrasy is largely attributable to a fluke of Lie-group theory:
 the rotation group ${\rm SO}(4)$ is not a simple Lie group. Instead, its Lie algebra splits as a direct sum
 \begin{gather*}
{\mathfrak{so}}(4) \cong {\mathfrak{so}}(3)\oplus
 {\mathfrak{so}} (3),
\end{gather*}
 as a consequence of the fact that left- and right-multiplication by the unit quaternions ${\rm Sp}(1)$ belong to
 different subgroups of the rotation group. On an oriented Riemannian $4$-manifold $\big(M^4, g\big)$, this gives rise
 to an invariant direct-sum decomposition
\begin{gather*}
\Lambda^2 = \Lambda^+ \oplus \Lambda^-
\end{gather*}
of the bundle of $2$-forms, because the action of ${\rm SO}(4)$ on $2$-forms is isomorphic, via index raising, to
its adjoint representation on the Lie algebra ${\mathfrak{so}}(4)$ of skew $4\times 4$ matrices. This decomposition
in fact coincides with the decomposition of the $2$-forms into the $(\pm)$-eigenspaces of the Hodge star operator
\begin{gather*}
\star\colon \ \Lambda^2 \to \Lambda^2.
\end{gather*}
We will now emphasize our choice of an orientation by focusing on the bundle $\Lambda^+$ of
{\em self-dual $2$-forms} $\varphi$, which are characterized by the condition $\star \varphi = \varphi$.

While the rank-$3$ oriented vector bundle $\Lambda^+\to M$ depends on the conformal class $[g]$ of the Riemannian metric
$g$, the bundles $\Lambda^+_g$ and $\Lambda^+_{g^\prime}$ associated with two different metrics are nonetheless canonically
bundle-isomorphic via the natural identification $\Lambda^+_g = \Lambda^2/\Lambda^-_g$,
because we always have $\Lambda_{g^\prime}^+ \cap\Lambda_g^-= 0$. This algorithm for producing an isomorphism
suffers from some defects, though. First of all, interchanging $g$ and $g^\prime$ does not produce the inverse isomorphism. Second, the isomorphism produced by this algorithm does not preserve the relevant inner products.
Fortunately, however, the latter can be corrected by applying a unique positive, self-adjoint endomorphism to $\Lambda^+$,
and this then allows us to identify the oriented bundles-with-inner-product~$\Lambda^+$ for two different metrics in a manner
that is unique up to isotopy. This will suffice to give a metric-independent meaning to the notions that are the main focus of our discussion.

We now fix a Riemannian metric $g$ on our oriented $4$-manifold $M$, and notice that, since
 ${\rm SO}(4)/\ZZ_2 \cong {\rm SO}(3)\times {\rm SO}(3)$, the $4$-dimensional rotation
 group acts transitively on the unit sphere in $\Lambda^+$. For this reason, any $\omega \in \Lambda^+_x$, $x\in M$,
 with $|\omega | = \sqrt{2}$ can be expressed as
 \begin{gather*}
\omega =e^1\wedge e^2 + e^3 \wedge e^4
\end{gather*}
 in some oriented orthonormal basis for $T_xM$, and hence corresponds, via index raising, to the
 endomorphism $\jmath\colon T_xM \to T_xM$ represented by the matrix
 \begin{gather*}
\left[\begin{array}{cccc} & -1 & & \\ 1 & & & \\ & & & -1 \\ & & 1 & \end{array}\right].
\end{gather*}
In other words, any such $\omega$ defines an almost-complex structure $\jmath$ at $x$ that is compatible with the metric
$g$ and determines the given orientation. Conversely, if $\jmath\colon T_xM\to T_xM$ satisfies
$\jmath^2 = -I$ and $\jmath^*g=g$, and also determines the given orientation of $M$, then $\jmath$ arises, via index raising,
 from a unique
$\omega \in \Lambda^+_x$ with $|\omega |= \sqrt{2}$.

We now define the {\em twistor space} of our oriented Riemannian $4$-manifold $(M,g)$ to be the total space
\begin{gather*}
Z := S_{\sqrt{2}}\big(\Lambda^+\big)= \big\{ \omega \in \Lambda^+\,|\, |\omega|= \sqrt{2}\big\}
\end{gather*}
of the $2$-sphere bundle $\wp\colon Z\to M$ associated with the oriented
rank-$3$ vector bundle $\Lambda^+\to M$ of self-dual $2$-forms. We may then give $Z$ an almost-complex structure
$J\colon TZ\to TZ$, $J^2=-I$, by the following construction, which is essentially due to Atiyah--Hitchin--Singer~\cite{AHS},
and which provides a general Riemannian context for Penrose's non-linear graviton construction~\cite{pnlg}.
We begin by decomposing $TZ$ into vertical and horizontal components,
\begin{gather}\label{horizon}
TZ = \mathsf{V}\oplus \mathsf{H},
\end{gather}
where $\mathsf{V} := \ker {\rm d}\wp$ and $\mathsf{H}$ is induced by parallel transport in $\Lambda^2$ with respect to the
Riemannian connection of $g$.
Now notice that index raising gives us an alternative, conformally invariant description
\begin{gather*}
Z= \big\{ {\mathsf \jmath}\colon T_xM\to T_xM, \,x\in M\,|\, {\mathsf \jmath}^2 =-I, \,{\mathsf \jmath}^* g = g, \,{\mathsf \jmath} >0\big\}
\end{gather*}
of the twistor space. Since the derivative
 ${\rm d}\wp\colon TZ \to TM$ of the bundle projection~$\wp$ induces a~tautological isomorphism~$\mathsf{H}\cong \wp^*TM$,
 we may therefore define
an endomorphism $J_{\mathsf{H}}\colon \mathsf{H}\to \mathsf{H}$ whose action
at $\jmath\in \wp^{-1}(x)$ is the horizontal lift of~$\jmath\colon T_xM\to T_xM$.
Meanwhile, since each fiber~$\wp^{-1}(x)$ of~$\wp\colon Z\to M$
is a round $2$-sphere in an oriented $3$-dimensional inner-product space~$\Lambda^+_x$, we can therefore
declare $J_\mathsf{V}\colon \mathsf{V}\to \mathsf{V}$ to be $+90^\circ$ rotation in the tangent space of each fiber $2$-sphere with respect to the {outward-pointing} orientation. Since these recipes guarantee
that $J_{\mathsf{H}}^2=-I_{\mathsf{H}}$ and $J_{\mathsf{V}}^2=-I_{\mathsf{V}}$,
setting
\begin{gather*}
J:= J_{\mathsf{H}}\oplus J_{\mathsf{V}}
\end{gather*}
now produces an almost-complex structure $J$ on $Z$,
 as illustrated by Figure~\ref{one}.
 \begin{figure}[htb]\centering
\includegraphics[scale=.4]{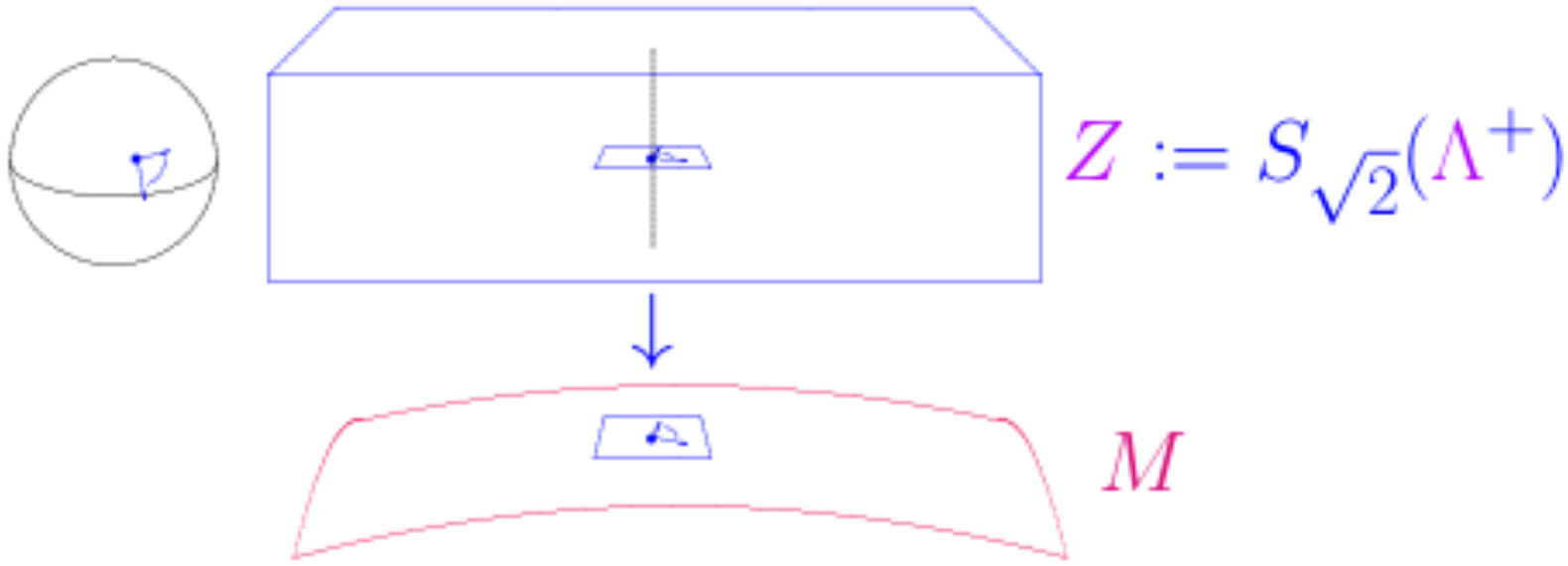}

\caption{The twistor space $Z$ of an oriented Riemannian $4$-manifold $(M,g)$ is the sphere bundle of the oriented rank-$3$ bundle
$\Lambda^+\to M$ of self-dual $2$-forms. This $6$-manifold can be given a canonical almost-complex structure $J$ that is compatible
with the decomposition of $TZ$ into vertical and horizontal subspaces.}\label{one}
\end{figure}

The almost-complex structure $J$ is actually conformally invariant, even though the
decomposition~\eqref{horizon} is not; however, $J$ is only integrable~\cite{AHS,pnlg}
if the Weyl curvature of $(M,g)$ is anti-self-dual. Nonetheless, some useful remnants of integrability persist, even in the general
case; in particular, each fiber $\wp^{-1} (x)\cong \CP_1$ is a $J$-holomorphic curve, and the normal bundle
$\nu = T^{1,0} Z/\mathsf{V}^{1,0}$ of each fiber is a holomorphic bundle $\nu \cong \mathcal{O}(1) \oplus \mathcal{O}(1)$
over this $\CP_1$.

Since this emphasizes the fact that $\wp\colon Z\to M$ may be thought of as a $\CP_1$-bundle,
it seems natural to ask whether this $\CP_1$-bundle can always be expressed as
 the projectivization $\mathbb{P} (\mathbb{V}_+)$
of a rank-$2$ complex vector bundle $\mathbb{V}_+\to M$. As we will see, the answer always turns out to be, ``Yes!''
This assertion
 exactly amounts to the
fact that any oriented $4$-manifold admits spin$^c$ structures. The road that will bring us to this realization begins with the following
definition:

 \begin{geomdef}A {\em spin$^c$ structure} on a connected oriented Riemannian $4$-manifold $(M,g)$ is a complex line bundle ${\mathscr{L}}\to {Z}$
on the twistor space
that has degree $1$ on any ${S^2}$ fiber of ${Z}\to M$.
\end{geomdef}

Here, two isomorphic complex line bundles on $Z$ are considered to define the same spin$^c$ structure. However,
 the first Chern class $c_1$ defines a bijection between equivalence classes of complex line bundles on $Z$
and $H^2(Z, \ZZ)$, in a manner that converts the tensor-product of line bundles into the
addition of cohomology classes. This means that the above definition
can be reformulated as saying that a spin$^c$ structure on $M$ is a cohomology class
 $\mathbf{a} \in H^2(Z,\ZZ)$ with $\langle \mathbf{a}, F \rangle =1$, where
$F \in H_2 (Z, \ZZ )$ is the homology class of a fiber $\wp^{-1} (x) \approx S^2$.

The above should be compared and contrasted with the following:

\begin{standef}
A {\em spin$^c$ structure} on an oriented Riemannian $4$-manifold $(M,g)$ is a~circle bundle
$\widehat{\mathfrak{F}}\to \mathfrak{F}$
over the oriented orthonormal frame bundle that is also
compatibly endowed with the structure of a
 principal ${\rm Spin}^c(4)$-bundle, where
 \begin{gather*}
{\rm Spin}^c(4):= [{\rm Sp}(1)\times {\rm Sp}(1)\times {\rm U}(1)]/\langle (-1,-1, -1)\rangle .
\end{gather*}
\end{standef}

Fortunately, these competing definitions are exactly equivalent.
For example, one can pass from the geometric definition to the standard definition by first expressing
$Z$ as $\mathfrak{F} /{\rm U}(2)$. Thus, if $\mathscr{L}\to Z$ is a complex line bundle of fiber degree $1$,
we can pull the circle bundle $S(\mathscr{L})$ back to obtain a circle bundle over $\mathfrak{F}$,
and this pull-back can then be made into principal ${\rm Spin}^c(4)$-bundle over $M$
using the fact that $H^1 ({\rm SO}(4), \mathcal{E}^\times )= H^1 ({\rm SO}(4), \CC^\times )=H^2( {\rm SO}(4), \ZZ) =\ZZ_2$.

In the opposite direction, given a principal ${\rm Spin}^c(4)$-bundle over $M$ that is also a circle bundle
$\widehat{\mathfrak{F}}\to {\mathfrak{F}}$, we first construct a
vector bundle $\mathbb{V}_+\to M$ by applying the associated bundle construction to the representation
${\rm Spin}^c(4)\to {\rm U}(2)\cong [{\rm Sp}(1) \times {\rm U}(1)]/\langle (-1,-1)\rangle$
obtained by dropping the second ${\rm Sp}(1)$. The map $\widehat{\mathfrak{F}}\to \mathfrak{F}$
then allows us to identify $Z$ with $\mathbb{P} (\mathbb{V}_+)$, and the $\mathcal{O}(1)$-bundle dual to
the tautological line bundle of $\mathbb{P} (\mathbb{V}_+)$ then provides the fiber-degree-$1$ line bundle featured in the
geometric definition.

While every oriented $4$-manifold will turn out to admit spin$^c$ structures, the situation is entirely different for
{\em spin structures}, which are much more restrictive:

\begin{geomdef}
A {\em spin structure} on an oriented Riemannian $4$-manifold $(M,g)$ is a~square-root ${\mathscr{V}^{1/2}}$ of the vertical complex line bundle $\mathscr{V}:= \mathsf{V}^{1,0}$
of the twistor space ${Z}\to M$.
\end{geomdef}

Here, a square-root $\mathscr{V}^{1/2}$ of $\mathscr{V}$ means a line-bundle $\mathscr{L}\to Z$ that is equipped
with a specific isomorphism $\Phi\colon \mathscr{L}\otimes \mathscr{L} \to \mathscr{V}$. Two such square-roots $(\mathscr{L}, \Phi)$ and
 $(\mathscr{L}^\prime, \Phi^\prime)$ are considered to be the same if there is an isomorphism $\Psi\colon \mathscr{L}\to \mathscr{L}^\prime$
 of complex line bundles that induces a~commutative diagram
\setlength{\unitlength}{1.5mm}
\begin{center}\begin{picture}(20,17)(0,3)
\put(2,17){\makebox(0,0){$\mathscr{L}\otimes \mathscr{L}$}}
\put(2,5){\makebox(0,0){$\mathscr{L}^\prime\otimes \mathscr{L}^\prime$}}
\put(21,11){\makebox(0,0){$\mathscr{V}$}}
\put(12,16){\makebox(0,0){$_\Phi$}}
\put(12,7){\makebox(0,0){$_{\Phi^\prime}$}}
\put(-1,11){\makebox(0,0){$_{\Psi\otimes \Psi}$}}
\put(7,16){\vector(3,-1){12}}
\put(8, 7){\vector(3,1){11}}
\put(2,15){\vector(0,-1){8}}
\end{picture}\end{center}

\begin{standef}
A {\em spin structure} on an oriented Riemannian $4$-manifold $(M , g )$ is a~double cover
${\widetilde{\mathfrak{F}}}\to {\mathfrak{F}}$
of the principal ${\rm SO}(4)$-bundle of oriented orthonormal frames
by a principal ${\rm Spin}(4)$-bundle, where
\begin{gather*}
{\rm Spin}(4)= {\rm Sp}(1)\times {\rm Sp}(1).
\end{gather*}
\end{standef}

Once again, these competing definitions are exactly equivalent. Indeed, the ${\rm SO}(3)$-bundle
$\mathsf{F}= \mathfrak{F}/{\rm Sp}(1)$ may be viewed as the circle bundle~$S( \mathscr{V})$ over
$Z= \mathsf{F}/{\rm SO}(2)$, so a square-root of~$ \mathscr{V}$ amounts to a double-cover of~$\mathsf{F}$.
On the other hand, the quotient map $\mathfrak{F}\to \mathsf{F}$ induces an isomorphism of fundamental groups,
so double covers of $\mathsf{F}$ are in bijective correspondence with double covers of~$\mathfrak{F}$.

\section{Oriented rank-3 vector bundles}\label{prelims}

We now focus on the case where our smooth oriented $4$-manifold is {\em compact}. The {\em Euler class}
$\mathsf{e} (\Lambda^+)$ will then play the starring role in our
proofs. However, our approach will be clarified by initially working in a more general context. Let us therefore first
 consider the Euler class $\mathsf{e} (\mathsf{E}) \in H^3(M, \ZZ)$
of any oriented rank-$3$ bundle $\mathsf{E}\to M$.

Now, since the rank of $\mathsf{E}$ is odd, multiplication by $-1$ induces an orientation-reversing
self-isomorphism $\mathsf{E}\to \mathsf{E}$. Hence the Euler class $\mathsf{e}= \mathsf{e}(\mathsf{E})\in H^3 (M,\ZZ)$ satisfies
 $\mathsf{e}= -\mathsf{e}$, and so
is a $2$-torsion element. This Euler class is Poincar\'e dual to the homology class $\in H_1(M, \ZZ)$ of the oriented curve cut out by the
zero set of any section of $\mathsf{E}\to M$ that is transverse to the zero section $\mathsf{0}_M\subset \mathsf{E}$.
Of course, if $\mathsf{E}$ had a nowhere-zero section, its Euler class
$\mathsf{e}(\mathsf{E})$ would consequently vanish. However,
 since $\rank \mathsf{E} < \dim M$, the converse is definitely false! This common mistake seems to
arise from over-familiarity with the Poincar\'e--Hopf theorem, which concerns the case where
 the rank of the bundle equals the dimension of the base; in that context,
generic zeroes are just isolated points, and generic
zeroes of opposite sign can then be eliminated in pairs. In our case, the zero locus of a generic section is instead a union of
oriented circles, and the vanishing of the Euler class just means that this curve bounds an oriented surface.

\begin{xpl} Since $H^3\big(S^4, \ZZ\big)=0$, the oriented rank-$3$ bundle $\Lambda^+\to S^4$ certainly must have
 $\mathsf{e} (\Lambda^+)=0$. However, non-zero global sections of this bundle
certainly {\em do not exist}. Otherwise, the twistor projection
$Z\to S^4$ would admit a smooth global section,
and we could then interpret this as an orientation-compatible almost complex structure
$J$ on $S^4$. However, $S^4$ does not admit an almost-complex structure! For example,
if it did, the index of the spin$^c$ Dirac operator generalizing $\bar{\partial} + \bar{\partial}^*$
would be $\todd \big(S^4\big) = \frac{\chi + \tau}{4} \big(S^4\big) = \frac{1}{2}\not\in \ZZ$; for further discussion, see the commentary following~\eqref{doc} below. This contradiction shows that every smooth section of \mbox{$\Lambda^+\to S^4$} must have non-empty
zero locus, even through $\mathsf{e} (\Lambda^+)=0$. \end{xpl}

We will show in Section~\ref{roundup} that $\mathsf{e} (\Lambda^+)=0$ on any
smooth oriented $4$-manifold $M$, even if $H^3( M, \ZZ)\neq 0$. In order to do this, however,
it will help to first put the question into the broader context of rank-$3$ oriented vector bundles
on compact oriented $4$-manifolds. Here, our approach will crucially depend on the following technical result:

\begin{Proposition} \label{packet}
Let $M$ be a smooth connected compact oriented $4$-manifold, let $\mathsf{E}\to M$ be a~real oriented rank-$3$ vector bundle,
 let $\varpi\colon \mathcal{Z}\to M$ be the
unit $2$-sphere bundle $\mathcal{Z}= S(\mathsf{E})$ with respect to some positive-definite inner product, and let $F\in H_2((\mathcal{Z}), \ZZ)$ denote the homology class of an $S^2$-fiber of $\varpi$.
Then the following conditions are all equivalent:
\begin{enumerate}\itemsep=0pt
\item[$(i)$] The Euler class $\mathsf{e}(\mathsf{E}) \in H^3(M, \ZZ)$ vanishes; 
\item[$(ii)$] There is a cohomology class $\mathbf{a}\in H^2 (\mathcal{Z}, \ZZ)$ with $\langle \mathbf{a}, F\rangle =1$; 
\item[$(iii)$] $H_2(\mathcal{Z}, \ZZ) \cong H_2 (M, \ZZ ) \oplus \ZZ$; and 
\item[$(iv)$] $|\mathfrak{T}_2 (\mathcal{Z})| = |\mathfrak{T}_2 (M)|$, where $\mathfrak{T}_k$ is the torsion subgroup of
 $H_k(\underline{\hphantom{M}}, \ZZ)$. 
\end{enumerate}
\end{Proposition}
\begin{proof}
Let us first recall that the cohomology of
$\mathcal{Z}$ is related to that of $M$ by the Gysin exact sequence \cite[Section~12]{milnorstaf}
\begin{gather}\label{geyser}
 \cdots H^{k-3}(M)\stackrel{\cup \mathsf{e}}{\to} H^{k}(M)\stackrel{\varpi^*}{\to} H^{k}(\mathcal{Z})\stackrel{\varpi_*}{\to}
 H^{k-2}(M) \stackrel{\cup
 \mathsf{e}}{\to} H^{k+1}(M) \cdots,
\end{gather}
where $\mathsf{e}= \mathsf{e}( \mathsf{E})$ is the Euler class of the oriented rank-$3$ bundle $\mathsf{E}$.
This is really just a disguised form of the long exact sequence
\begin{gather*}
 \cdots \to H^{k}(\mathsf{E}, \mathsf{E}-M){\to} H^{k}(\mathsf{E}){\to} H^{k}(\mathsf{E}-M){\to}
 H^{k+1}(\mathsf{E}, \mathsf{E}-M) \to \cdots
\end{gather*}
of the pair $(\mathsf{E}, \mathsf{E}-M)$, because the zero section $M\hookrightarrow \mathsf{E}$ is a
deformation retract of $\mathsf{E}$, and $\mathcal{Z}$ is a deformation retract of $\mathsf{E}-M$; the Thom isomorphism
$H^{k-3} (M) \to H^{k} ( \mathsf{E}, \mathsf{E}-M)$, given by cupping with the Thom class, therefore
converts one exact sequence into the other.
While this works perfectly well with coefficients in any ring, we will actually only use it for $\ZZ$, $\ZZ_2$, and $\RR$ coefficients in this article.

First observe that the exactness of the
Gysin sequence
\begin{gather*}
\cdots \to H^2 (Z, \ZZ ) \stackrel{\varpi_*}{\to} H^0 (M, \ZZ ) \stackrel{\mathsf{e} \cdot}{\to}H^3 (M, \ZZ ) \to \cdots
\end{gather*}
implies that (i) $\Longrightarrow$ (ii), because the vanishing of $\mathsf{e}= \mathsf{e}(\mathsf{E})$
implies the surjectivity of the map
$\varpi_*\colon H^2( \mathcal{Z} , \ZZ )\to H^0 (M, \ZZ )=\ZZ$ given by $\mathbf{a} \mapsto \langle \mathbf{a}, F\rangle$.

Next, since $H^4(M, \ZZ)= \ZZ$ is free, while $\mathsf{e} (\mathsf{E})$ is a torsion class, observe that
the terminal segment of the Gysin sequence~\eqref{geyser} breaks off as the short exact sequence
\begin{gather*}
0\to H^4(M, \ZZ) \stackrel{\varpi^*}{\to} H^4 (\mathcal{Z}, \ZZ) \stackrel{\varpi_*}{\to} H^2 (M, \ZZ) \to 0,
\end{gather*}
where the image of $H^4(M,\ZZ)$ in $H^4 (\mathcal{Z},\ZZ)$ is generated by the Poincar\'e dual of the
fiber class~$F$. Applying Poincar\'e duality in both $\mathcal {Z}$ and $M$ therefore converts this into an exact sequence
\begin{gather}\label{homology}
0\to\ZZ \stackrel{\cdot F}{\to} H_2 (\mathcal{Z}, \ZZ) {\to} H_2 (M, \ZZ) \to 0.
\end{gather}
If statement (ii) holds, then there is
 an $\mathbf{a}\in H^2(\mathcal{Z}, \ZZ )$ with $\langle \mathbf{a} , F\rangle =1$, and pairing with $\mathbf{a}$
 then provides a left inverse of $\ZZ \hookrightarrow H_2 (\mathcal{Z}, \ZZ)$.
Thus (ii) implies that the exact
sequence~\eqref{homology} {\em splits}, and so yields an isomorphism
\begin{gather}
\label{splitsville}
H_2(\mathcal{Z}, \ZZ) \cong \ZZ\oplus H_2 (M, \ZZ ).
\end{gather}
This shows that (ii) $\Longrightarrow$ (iii).

Next, observe that whenever \eqref{splitsville} holds, one also has
\begin{gather*}
\mathfrak{T}_2 (\mathcal{Z} ) \cong \mathfrak{T}_2 (M),
\end{gather*}
where $\mathfrak{T}_k$ denotes the torsion subgroup of the integer homology $H_k$.
Since this in particular implies that these torsion subgroups have the same order,
it follows that (iii) $\Longrightarrow$ (iv).

On the other hand, the universal coefficients theorem tells us that
\begin{gather}\label{judy}
\mathfrak{T}^3 (\mathcal{Z} ) \cong \mathfrak{T}_2 (\mathcal{Z} ) \qquad \mbox{and} \qquad \mathfrak{T}^3 (M) \cong \mathfrak{T}_2 (M),
\end{gather}
where $ \mathfrak{T}^k = \Ext ( \mathfrak{T}_{k-1} , \ZZ) \cong \mathfrak{T}_{k-1}$ denotes the torsion subgroup of the integer cohomology $H^k$.
Thus, if $|\mathfrak{T}_2 (\mathcal{Z} ) | = |\mathfrak{T}_2 (M)|$, it then follows that $\big|\mathfrak{T}^3 (\mathcal{Z} )\big| = \big|\mathfrak{T}^3 (M)\big|$.
However, the central portion of the Gysin sequence reads
\begin{gather*}
\cdots \to H^0 (M, \ZZ) \stackrel{\mathsf{e}\cdot}{\to} H^3 (M, \ZZ) \to H^3 (\mathcal{Z}, \ZZ ) \to H^1(M, \ZZ)\to \cdots,
\end{gather*}
where $ \mathsf{e}= \mathsf{e}(\mathsf{E})\in H^3 (M, \ZZ)$ is a $2$-torsion class. Since $H^1(M, \ZZ)$ is torsion-free,
this means, in particular, that
\begin{gather}\label{punch}
\mathfrak{T}^3 (\mathcal{Z} )\cong \begin{cases}
 \mathfrak{T}^3 (M ) & \text{if } \mathsf{e}(\mathsf{E})=0, \\
 \mathfrak{T}^3 (M )/\ZZ_2 & \text{otherwise}.
\end{cases}
\end{gather}
Thus, $\big|\mathfrak{T}^3 (\mathcal{Z} )\big| = \big|\mathfrak{T}^3 (M)\big|$ only when $\mathsf{e}(\mathsf{E})=0$.
This proves that (iv) $\Longrightarrow$ (i). We have thus shown that
\begin{gather*}
 {\rm (i)} \Longrightarrow {\rm (ii)} \Longrightarrow {\rm (iii)} \Longrightarrow {\rm (iv)} \Longrightarrow {\rm (i)},
\end{gather*}
and our proof of the proposition is therefore complete. \end{proof}

\section[Existence of spin c structures]{Existence of spin$\boldsymbol{{}^c}$ structures}\label{roundup}

Proposition \ref{packet} and some twistor geometry now imply the following:

\begin{Theorem}\label{existence} Any smooth compact oriented $4$-manifold admits spin$^c$ structures.
\end{Theorem}
\begin{proof}
For any Riemannian metric $g$ on $M$, our geometric definition of a spin$^c$ structure restates the claim as asserting the existence of
a cohomology class $\mathbf{a}\in H^2(Z, \ZZ)$ with $\langle \mathbf{a} , F\rangle =1$,
where $Z= S(\Lambda^+)$ is the twistor space, and $F= \big[S^2\big]$ is the fiber homology class. The equivalence
(ii) $\Longleftrightarrow$ (iv) in Proposition~\ref{packet} therefore says that it suffices to
 show that $|\mathfrak{T}_2 (Z)| = |\mathfrak{T}_2 (M)|$, where $\mathfrak{T}_2$ denotes the torsion subgroup of the
integer homology $H_2$. Since~\eqref{judy} and~\eqref{punch} also guarantee that
$|\mathfrak{T}_2 (Z)| \leq |\mathfrak{T}_2 (M)|$, it therefore suffices to show that the homomorphism
$\wp_* \colon \mathfrak{T}_2 (Z)\to \mathfrak{T}_2 (M)$ is surjective.

To see this, we begin by recalling that any element of $H_2(M,\ZZ)$ can be represented by a~smooth compact connected embedded oriented surface $\Sigma^2 \subset M^4$. Indeed, any homology class
$[\Sigma ]\in H_2 (M, \ZZ )$ is Poincar\'e dual to an element of~$H^2(M, \ZZ)$, which can then be
realized as the first Chern class of a complex line bundle on $M$. A generic smooth section of
this line bundle then provides a smooth compact oriented embedded surface ${\Sigma}_0$ representing~$[\Sigma]$.
This representative might still not be connected, but we can then correct this by connecting the various components of~${\Sigma}_0$
by narrow tubes $I \times S^1$ that approximate a collection of disjoint arcs between the different connected components.

 Now suppose that $\Sigma\subset M$ is a smooth compact connected oriented surface representing a~torsion class
 $[\Sigma ] \in \mathfrak{T}_2 (M)\subset H_2 (M , \ZZ)$. Since the homological self-intersection
 $[\Sigma ] \bullet [\Sigma ]$ of a~torsion class must vanish, it therefore follows that the normal bundle $\mathfrak{N}$ of $\Sigma\subset M$
 has trivial Euler class, and is therefore trivial. We now define an orientation- and $g$-compatible almost-complex structure $\jmath$
 on $TM|_\Sigma= \mathfrak{N} \oplus T\Sigma$ by declaring it to be given by $90^\circ$ rotation in both $\mathfrak{N}$ and $T\Sigma$.
 Since this $\jmath$ defines a section of $Z|_\Sigma$, its image defines an embedded surface $\hat{\Sigma}\subset Z$
 that projects diffeomorphically to~$\Sigma$ via the twistor projection. However, the twistor almost-complex structure $J_{\mathsf{H}}$
 on the horizontal bundle $\mathsf{H}\subset TZ$ then restricts to $\hat{\Sigma}$ as an exact copy of the action of
 $\jmath$ on $TM|_\Sigma$. Since the normal bundle $\mathfrak{N}$ of $\Sigma$ is trivial, we thus have
 \begin{gather*}
\big\langle c_1 (\mathscr{H}) , [\hat{\Sigma}]\big\rangle = \big\langle c_1 (\CC \oplus T^{1,0}\Sigma ) , [{\Sigma}]\big\rangle = \chi (\Sigma ) = 2(1-\mathfrak{g}),
\end{gather*}
where $\mathfrak{g}$ denotes the genus of $\Sigma$ and where, once again, $\mathscr{H} = \mathsf{H}^{1,0}$.

Now consider the homology class $A= [\hat{\Sigma}] + (\mathfrak{g} -1) F \in H_2 (Z, \ZZ)$, where $F$ once again denotes
 the homology class of a fiber of the twistor projection $\wp\colon Z\to M$. Since $\wp$ induces an oriented diffeomorphism $\hat{\Sigma} \to \Sigma$,
 and since $\wp$ collapses $F$ to a point, $\wp_*\colon H_2(Z, \ZZ) \to H_2(M, \ZZ)$ therefore sends $A\mapsto [\Sigma ]$.
We will now prove that $A$ is a torsion class. To show this, it suffices to check that $\langle \mathbf{a} , A \rangle =0$
 for every $\mathbf{a} \in H^2(Z, \RR)$. However, the $\RR$-coefficient version of the Gysin sequence~\eqref{geyser}
 implies that $H^2 (Z, \RR)= \wp^* H^2(M, \RR)\oplus
 \RR c_1 (\mathscr{H})$, since $\langle c_1 (\mathscr{H}) , F \rangle =2\neq 0$.
But pairing $A$ with an element of $\wp^* H^2(M, \RR)$ amounts to pairing $\wp_*(A)=[\Sigma ]$ with an element of $H^2(M, \RR)$,
which yields zero because $[\Sigma ]$ is a torsion class by hypothesis. But because
 the restriction of $\mathscr{H}$ to an~$S^2$ fiber has degree $+2$, we also have
\begin{gather*}
 \langle c_1 (\mathscr{H}) , A\rangle = \big\langle c_1 (\mathscr{H}) , [\hat{\Sigma}]\big\rangle + (\mathfrak{g} -1) \langle c_1 (\mathscr{H}) , F \rangle
 = 2(1-\mathfrak{g}) + (\mathfrak{g} -1) 2 =0.
 \end{gather*}
This shows that $A$ is a torsion class in $H_2(Z, \ZZ)$ with $\wp_* (A) = [\Sigma ]$,
and therefore proves that $\wp_*\colon \mathfrak{T}_2 (Z)\to \mathfrak{T}_2 (M)$ is surjective, as claimed.
 \end{proof}

Applying Proposition \ref{packet} to the Gysin sequence, we thus have:

\begin{Corollary}\label{sharpie}
Any compact oriented Riemannian $4$-manifold $(M,g)$ satisfies $\mathsf{e} (\Lambda^+)=0$.
Consequently, with either $\ZZ$ or $\ZZ_2$ coefficients,
\begin{gather*}
0\to H^k (M) \stackrel{\wp^*}{\to} H^k (Z) \stackrel{\wp_*}{\to} H^{k-2}(M)\to 0
\end{gather*}
is exact for every $k$, where $\wp\colon Z\to M$ is the twistor projection.
\end{Corollary}

As an application, we therefore have the following:
\begin{Proposition} One has $w_2(TM)= w_2(\Lambda^+)$
for every compact oriented Riemannian four-manifold $(M,g)$. Moreover,
 $w_2(TM)\in H^2(M, \ZZ_2)$ is the mod-$2$ reduction of some integer cohomology class $\mathbf{c}\in H^2(M, \ZZ)$.
\end{Proposition}
\begin{proof}
We have $\wp^* w_2 (TM) = w_2 (\wp^* TM) = w_2 (\mathscr{H}) = \rho [ c_1 (\mathscr{H})] =
 \rho \big[ c_1 \big({\wedge}^{2} \mathscr{H}\big)\big]=
w_2 \big({\wedge}^{2} \mathscr{H}\big)\allowbreak = w_2\big(\RR \oplus \wedge^{2} \mathscr{H}^*\big) = w_2 (\wp^* \Lambda^+) = \wp^* w_2 (\Lambda^+)$.
Since $\wp^*\colon H^2(M, \ZZ_2)\to H^2(Z, \ZZ_2)$ is injective by Corollary~\ref{sharpie}, this implies that $w_2(TM)= w_2(\Lambda^+)$.

Now consider the third {integer} Stiefel--Whitney class defined by
$W_3 := \beta (w_2)$,
where~$\beta$ is the Bockstein homomorphism of the long exact sequence
\begin{gather}\label{promotion}
\cdots \to H^2 (\underline{\hphantom{M}}, \ZZ ) \stackrel{2\cdot}{\to} H^2 (\underline{\hphantom{M}}, \ZZ ) \stackrel{\rho}{\to} H^2 (\underline{\hphantom{M}}, \ZZ_2 ) \stackrel{\beta}{\to} H^3 (\underline{\hphantom{M}}, \ZZ )\to \cdots
\end{gather}
induced by
\begin{gather*}
0\to \ZZ \stackrel{2\cdot}{\to} \ZZ\stackrel{\rho}{\to} \ZZ_2\to 0,
\end{gather*}
with $\rho$ denoting reduction mod $2$.
Since $w_2 (\wp^* TM) = \rho [ c_1 (\mathscr{H})]$, it follows that $\wp^* W_3 (TM)= W_3 (\wp^* TM) =\beta [w_2(\wp^* TM) ]=0$.
The injectivity of $\wp^*\colon H^3(M, \ZZ)\to H^3(Z, \ZZ)$ guaranteed by Corollary \ref{sharpie} thus implies that $W_3(TM)= \beta [w_2(TM) ]$ vanishes,
and the exactness of~\eqref{promotion} then tells us that $w_2(TM)$ belongs to the image of $\rho\colon H^2(M,\ZZ) {\to} H^2(M,\ZZ_2)$.
\end{proof}

It is now easy to deduce an analogous result for the simplest non-compact $4$-manifolds:

\begin{Corollary}\label{intern}
 Let $M$ be the interior of a smooth compact oriented $4$-manifold-with-boundary. Then $M$ admits spin$^c$ structures.
\end{Corollary}
\begin{proof} If $M$ is displayed as the interior of a compact oriented $4$-manifold-with-boundary $X$,
we first construct the double
 $N= X\cup_{\partial X} \overline{X}$ of $X$, where $\overline{X}$ denotes $X$ equipped with the
opposite orientation. Each component of $N$ then admits a spin$^c$ structure by Theorem \ref{existence},
so taking a~union over components
 gives us a complex line bundle $\mathscr{L}\to S(\Lambda^+)$, where
 $\Lambda^+\to N$ is defined with respect to some Riemannian metric on~$N$. Since $M$ is an open subset of
 $N$, the twistor space of $M$ is an open subset of the twistor space of $N$, and restricting $\mathscr{L}\to S(\Lambda^+)$
 to this subset now gives us a spin$^c$ structure on~$M$. Of course, this construction is carried out with respect to a Riemannian metric
on $M$ that happens to arises by
 restriction from $N$, but the conclusion does not depend on a choice of metric, since the bundle-isomorphism class of
$\Lambda^+\to M$ is actually metric-independent. \end{proof}

There seems to be a widespread consensus \cite{gost,krees} that this result should also hold for general
non-compact $4$-manifolds, including those that are not homotopy-equivalent to finite cell complexes.
Unfortunately, however, a watertight proof of this assertion is currently lacking.
For example, it does not suffice to exhaust $M$ by precompact regions
 $M_1 \subset M_2 \subset \cdots \subset M_j \subset \cdots$ with smooth boundary, and then
observe that $\mathsf{e} (\Lambda^+)=0$ vanishes on each $M_j$ by Corollary~\ref{intern},
because $H^* (M,\ZZ)\neq \varprojlim H^*(M_j,\ZZ)$
 in general; cf.\ \cite[p.~109]{milnorstaf}.

\section[Spin and spin c geometry]{Spin and spin$\boldsymbol{{}^c}$ geometry}\label{products}

While every compact oriented $4$-manifold admits spin$^c$ structures, such structures are typically far from unique:

\begin{Theorem} On any smooth compact oriented $4$-manifold $M$, the cohomology group $H^2(M, \ZZ)$ acts freely and transitively
on the set of spin$^c$ structures on $M$.
\end{Theorem}
\begin{proof} Assuming, for simplicity, that $M$ is connected, Corollary~\ref{sharpie}
tells us that there is an exact sequence
\begin{gather*}
0\to H^2 (M,\ZZ) \stackrel{\wp^*}{\to} H^2 (Z,\ZZ) \stackrel{\wp_*}{\to} \ZZ \to 0,
\end{gather*}
while our geometric definition tells us that $\{ \mbox{spin$^c$ structures on } M\}$ is exactly
$\wp_*^{-1} (1) \!\subset\! H^2 (Z,\ZZ)$. This corresponds to tensoring $\mathscr{L}\to Z$ with
pull-backs of line bundles on~$M$.
\end{proof}

By contrast, many $4$-manifolds do not admit spin structures:

\begin{Theorem} A smooth oriented $4$-manifold $M$ admits a spin structure
iff $w_2(TM)=0$. When this happens, $H^1(M, \ZZ_2)$ acts freely and transitively
on the spin structures of $M$.
\end{Theorem}
\begin{proof}
By our geometric definition, a spin structure exists iff the vertical line bundle $\mathscr{V}\to Z$
has a square root. However, as pointed out by Hitchin \cite{hitka}, there is a canonical isomorphism
$\mathscr{V} \cong \wedge^2\mathscr{H}$, because these two bundles have tautological identifications with the very same
line sub-bundle of $\CC\otimes \wp^*( \Lambda^+)$. Thus, $M$ admits a spin structure iff
$c_1 (\mathscr{H})= c_1 \big({\wedge}^2\mathscr{H}\big) = c_1(\mathscr{V})$ is divisible by $2$ in $H^2(Z,\ZZ)$.
But because
\begin{gather*}
\cdots \to H^2 (Z,\ZZ) \stackrel{2\cdot}{\to} H^2 (Z,\ZZ) \stackrel{\rho}{\to} H^2 (Z,\ZZ_2) \to \cdots
\end{gather*}
is exact, this happens iff $\rho [c_1 (\mathscr{H})] = w_2 (\mathscr{H}) = \wp^* w_2(TM)$
vanishes. Since $\wp^* \colon H^2(M, \ZZ_2) \to H^2(Z, \ZZ_2)$ is injective by Corollary~\ref{sharpie},
it therefore follows that $M$ admits a spin structure iff $w_2(TM)=0$.

When $w_2(TM)=0$,
the spin structures are exactly those double covers of the principal $\CC^\times$-bundle $\mathscr{V}^\times = \mathscr{V}-\mathsf{0}_Z$ that also doubly cover
the fiber; equivalently, they correspond to elements of $H^1(S(\mathscr{V}), \ZZ_2)$ that are non-zero on the
fiber circle. Since $w_2(TM)=0$ implies that $w_2(\mathscr{V})=0$, the Gysin sequence
of $\mathscr{V}\to Z$ then simplifies to yield
\begin{gather*}
0\to H^1 (Z , \ZZ_2) \to H^1(S(\mathscr{V}), \ZZ_2) \to H^0 (Z, \ZZ_2) \to 0
\end{gather*}
and it follows that $H^1 (Z , \ZZ_2)$ acts freely and transitively on such elements of $H^1(S(\mathscr{V}), \ZZ_2)$.
Since we also have $H^1(M, \ZZ_2) = H^1(Z,\ZZ_2)$ by Corollary \ref{sharpie},
$H^1(M, \ZZ_2)$ therefore acts freely and transitively on spin structures by tensoring $\mathscr{V}^{1/2}$ with pull-backs
of real line bundles. \end{proof}

We now give a direct twistorial construction of the spinor bundles of a spin structure, and twisted-spinor bundles of
a spin$^c$ structure. This provides yet another way of seeing that our geometric definitions of
such structures are equivalent to the standard definitions.

Given a square-root $\mathscr{V}^{1/2}$ of the vertical line bundle $\mathscr{V}\to Z$ on the twistor space,
we begin by noticing that $\mathscr{V}$ is canonically isomorphic to the tangent bundle $T^{1,0} \CP_1(x)$
on any twistor fiber $\CP_1 (x) : = \wp^{-1} (x)$, $x$. This gives $\mathscr{V}^{1/2}$ a natural fiber-wise
holomorphic structure, and we may therefore define two $2$-dimensional vector spaces at each $x\in M$ by
\begin{gather*}
{ \mathbb{S}}_{+{ x}}= H^0 \big({ \CP_{1}}({ x}) , \mathcal{O} \big(\mathscr{V}^{1/2}\big) \big), \qquad
{ \mathbb{S}}_{-{ x}}= H^0 \big({ \CP_{1}}({ x}) , \mathcal{O} \big({ \nu} \otimes \mathscr{V}^{-1/2}\big)\big),
\end{gather*}
where the normal bundle $\nu = \mathscr{H}|_{\CP_1 (x)}$ of $\CP_1(x) \subset Z$ is thought of as a holomorphic bundle $\cong \mathcal{O}(1) \oplus \mathcal{O}(1) $.
These naturally define smooth vector bundles $\mathbb{S}_\pm$, because we
may define the smooth sections of $\mathbb{S}_+\to M$ (respectively, $\mathbb{S}_-\to M$)
to be the smooth sections of $\mathscr{V}^{1/2}\to Z$
(respectively, ${ \nu} \otimes \mathscr{V}^{-1/2}\to Z$) that are holomorphic on each fiber of $\wp$.
These bundles naturally reduce to the structure group ${\rm SU}(2)={\rm Sp}(1)$, and the principal ${\rm Spin}(4)$-bundle
$\widetilde{\mathfrak{F}}\to M$ then arises as a bundle of adapted frames for $\mathbb{S}_+\oplus \mathbb{S}_-$.

If we instead start with a degree-$1$ complex line-bundle $\mathscr{L}\to Z$, the Gysin sequence allows us to choose an isomorphism
$\mathscr{L}^2= \mathscr{V}\otimes \wp^*\mathsf{L}$ for
a unique complex line-bundle $\mathsf{L}\to M$. Since~$\wp^*\mathsf{L}$ has a natural flat connection on each twistor fiber $\CP_1(x)$,
this gives ${ \mathscr{L}}= (\mathscr{V}\otimes \wp^*\mathsf{L})^{1/2}$ a~natural fiber-wise holomorphic structure.
We thus obtain a pair of vector bundles $\mathbb{V}_\pm \to M$ by setting
\begin{gather*}
{ \mathbb{V}}_{+{ x}}= H^0 \big({ \CP_{1}}({ x}) , \mathcal{O} ({ \mathscr{L}})\big), \qquad
{ \mathbb{V}}_{-{ x}}= H^0 \big({ \CP_{1}}({ x}) , \mathcal{O} ({ \nu} \otimes { \mathscr{L}}
\otimes {\mathscr{V}^*})\big).
\end{gather*}
On any spin subset of $M$, these bundles
can then be re-expressed as
$\mathbb{V}_\pm = \mathbb{S}_\pm\otimes \mathsf{L}^{1/2}$.
Moreover, $\wedge^2 \mathbb{V}_+ = \wedge^2 \mathbb{V}_-= \mathsf{L}$, and the principal
${\rm Spin}^c(4)$-bundle $\widehat{\mathfrak{F}}\to M$ arises as a bundle of adapted frames for $\mathbb{V}_+\oplus \mathbb{V}_-$.

In particular, any spin$^c$ structure gives us a rank-$2$ complex vector bundle $\mathbb{V}_+\to M$ such that
$\mathbb{P} (\mathbb{V}_+)= Z = S (\Lambda^+)$. However, because $\rank_\RR \mathbb{V}_+= 4 = \dim M$,
the Poincar\'e--Hopf paradigm applies, and can be used to predict the existence of non-zero sections of $\mathbb{bV}_+$.
If $M$ is compact, we can always choose a generic section of $\mathbb{V}_+$ that vanishes at only a finite number of points;
by following a suitable self-isotopy of $\mathbb{V}_+\to M$,
we can then arrange for all of these zeroes to be contained in an arbitrarily small ball
$B_\varepsilon (p) \subset M$,
and then use a local trivialization over this ball to alter this section so that it only vanishes at the center $p$ of the ball
(albeit typically with high multiplicity). Applying the projection $\mathbb{V}_+- \mathsf{0}_M \to \mathbb{P} (\mathbb{V}_+)= Z$,
we thus obtain a section of $\wp\colon Z- \wp^{-1}(p) \to M -\{ p \}$, and we may then interpret this section as an almost-complex structure~$J$ on~$M -\{ p \}$. Moreover, the image of this $J$ is a closed codimension-$2$ submanifold that is Poincar\'e dual to $c_1(\mathscr{L})$
for the given spin$^c$ structure, and so completely determines the spin$^c$ structure on $M-\{ p \}$. On the other hand,
removing a point from $M$ does not change~$H^2(\ZZ)$, and Corollary~\ref{sharpie} therefore tells us that
spin$^c$ structures on~$M$ are completely determined by their restrictions to $M-\{ p \}$. In summary, we have proved:

\begin{Theorem}Let $(M, g)$ be a compact connected oriented Riemannian $4$-manifold, and let
${p}\in M$ be an arbitrary base-point. Then
\begin{itemize}\itemsep=0pt
\item $M- \{ p\}$ admits almost-complex structures $J$ compatible with
the given metric and orientation;
\item any such $J$ determines a spin$^c$ structure on $M$; and
\item every spin$^c$ structure on $M$ arises this way.
\end{itemize}
\end{Theorem}

Similarly, on the interior of any compact oriented
$4$-manifold-with-nonempty-boundary, there always
exist almost-complex structures compatible with the given orientation. Every
such almost-complex structure moreover determines a spin$^c$ structure, and every spin$^c$ structure arises in this way.

However, in the compact case, the count of zeroes with multiplicity for a section of $\mathbb{V}_+$ is given by
the Euler number
\begin{gather}\label{doc}
 \int_M c_2 (\mathbb{V}_+) = \frac{c_1^2 (\mathsf{L}) - (2\chi + 3\tau )(M)}{4},
\end{gather}
where $\chi (M)$ and $\tau (M)$ respectively denote Euler characteristic and signature of $M$.
One can therefore find a global almost-complex structure on $M$ if and only if $c_1(\mathsf{L})\in H^2(M, \ZZ)$
can be chosen to make the right-hand side of \eqref{doc} vanish. This therefore happens~\cite{hiho} if and only if
there is some $\mathbf{c}\in H^2(M, \ZZ)$ with $\rho (\mathbf{c} ) = w_2(TM)$ and $\mathbf{c}^2 = (2\chi + 3\tau )(M)$.

\section{Other dimensions}

Our discussion has shown that a twistorial perspective can shed
new light on spin$^c$ geometry in dimension four.
We now conclude by pointing out some partial generalizations
of these ideas to other dimension.

If $(M,g)$ is any oriented Riemannian $2m$-manifold, its {\em twistor space} $Z$ is the fiber bundle
$Z = \mathfrak{F}/{\rm U}(m)$, where $\mathfrak{F}\to M$ is the principal ${\rm SO}(2m)$-bundle
of oriented orthonormal frames for~$TM$. Each fiber ${\digamma}$ of $Z$ is thus a copy of the homogeneous
 space ${\rm SO}(2m)/{\rm U}(m)$ of real dimension~$m(m-1)$. However, ${\digamma}$ is actually a compact
{\em Hermitian symmetric space}, and so is naturally a compact complex Fano manifold of complex
dimension $\mathsf{d}=\binom{m}{2}$. The twistor fiber $\digamma_x$ over $x\in M$ thus parameterizes the complex structures on $T_xM\cong \RR^{2m}$ that are
compatible with the given metric and orientation; equivalently, $\digamma$ is just the space
of $\alpha$-planes \cite[Appendix]{pr2} in $\CC\otimes T_xM\cong \CC^{2m}$.
Consequently, the twistor space $Z$ again admits a tautological almost-complex structure $J= J_\mathsf{H} \oplus J_\mathsf{V}$, allowing
us to view the horizontal and vertical subspaces of $TZ$ as complex vector bundles $\mathscr{H}= \mathsf{H}^{1,0}$
and $\mathscr{V}= \mathsf{V}^{1,0}$. Because there is a natural isomorphism $\mathscr{V} = \wedge^2 \mathscr{H}$,
the vertical anti-canonical line-bundle
$K^{-1}_\wp= \wedge^{\mathsf{d}}\mathscr{V}$ has a natural $(m-1)^{\rm st}$ root
$K^{-1/(m-1)}_\wp$ given by $\wedge^m\mathscr{H}$.
On the other hand, $c_1 \big(K^{-1/(m-1)}_\wp\big)$ restricts to any fiber as twice the generator of
$H^2 ({{\digamma}}, \ZZ) = \ZZ$.
 A spin structure on $M$ is just a
square root $K^{-1/2(m-1)}_\wp$ of
$K^{-1/(m-1)}_\wp$, while a spin$^c$ structure on $M$ is simply a line bundle on $Z$
whose Chern class restricts to a fiber ${{\digamma}}$ as the generator of $H^2({{\digamma}},\ZZ)=\ZZ$.
The corresponding spinor and twisted-spinor vector bundles can then be manufactured by a straightforward generalization of
the constructions described in Section~\ref{products}.

The odd-dimensional case is similar. If $(M,g)$ is an oriented Riemannian $(2m-1)$-manifold, with oriented orthonormal
frame bundle $\mathfrak{F}\to M$, its twistor space is defined to be $\mathfrak{F}/{\rm U}(m-1)$, and its fibers ${\digamma}= {\rm SO}(2m-1)/{\rm U}(m-1)\cong {\rm SO}(2m)/{\rm U}(m)$ are thus
identical to the fibers discussed above. The vertical anti-canonical line bundle still has a natural $(m-1)^{\rm st}$ root~$K^{-1/(m-1)}_\wp$, and a spin structure on $M$ is again just a square-root $K^{-1/2(m-1)}_\wp$ of this line-bundle.
A spin$^c$ structure is once again just
a line-bundle on~$Z$ of fiber-degree~$1$.

However, our Gysin-sequence approach to the $4$-dimensional case does not generalize to higher dimensions. This makes the use of the Leray--Serre spectral sequence absolutely essential for a~full understanding of the topological issues that arise in this broader context.

\vspace{-1mm}

\subsection*{Acknowledgements}
 This article is dedicated to my friend and teacher Sir Roger Penrose, in celebration of
his ninetieth birthday and recent Nobel Prize in Physics.
It is a pleasure to thank Dennis Sullivan for
his advice and encouragement, and Jiahao Hu for some very illuminating conversations. This research was supported in part by NSF grant DMS-1906267.

\vspace{-1mm}

\pdfbookmark[1]{References}{ref}


\begin{thebibliography}{99}
\footnotesize\itemsep=-1pt

\bibitem{AHS}
Atiyah M.F., Hitchin N.J., Singer I.M., Self-duality in four-dimensional
 {R}iemannian geometry, \href{https://doi.org/10.1098/rspa.1978.0143}{\textit{Proc. Roy. Soc. London Ser.~A}} \textbf{362}
 (1978), 425--461.

\bibitem{gost}
Gompf R.E., Stipsicz A.I., {$4$}-manifolds and {K}irby calculus,
 \textit{Graduate Studies in Mathematics}, Vol.~20, \href{https://doi.org/10.1090/gsm/020}{Amer. Math.
 Soc.}, Providence, RI, 1999.

\bibitem{hiho}
Hirzebruch F., Hopf H., Felder von {F}l\"achenelementen in 4-dimensionalen
 {M}annigfaltigkeiten, \href{https://doi.org/10.1007/BF01362296}{\textit{Math. Ann.}} \textbf{136} (1958), 156--172.

\bibitem{hitka}
Hitchin N.J., K\"ahlerian twistor spaces, \href{https://doi.org/10.1112/plms/s3-43.1.133}{\textit{Proc. London Math. Soc.}}
 \textbf{43} (1981), 133--150.

\bibitem{krees}
Killingback T.P., Rees E.G., {${\rm Spin}^c$} structures on manifolds,
 \href{https://doi.org/10.1088/0264-9381/2/4/008}{\textit{Classical Quantum Gravity}} \textbf{2} (1985), 433--438.

\bibitem{milnorstaf}
Milnor J.W., Stasheff J.D., Characteristic classes, \textit{Annals of
 Mathematics Studies}, Vol.~76, Princeton University Press, Princeton, N.J.,
 1974.

\bibitem{pnlg}
Penrose R., Nonlinear gravitons and curved twistor theory, \href{https://doi.org/10.1007/bf00762011}{\textit{Gen.
 Relativity Gravitation}} \textbf{7} (1976), 31--52.

\bibitem{pr2}
Penrose R., Rindler W., Spinors and space-time. {V}ol.~2. Spinor and twistor
 methods in space-time geometry, \textit{Cambridge Monographs on Mathematical Physics},
 \href{https://doi.org/10.1017/CBO9780511524486}{Cambridge University Press}, Cambridge, 1986.

\bibitem{whitney}
Whitney H., On the topology of differentiable manifolds, in Lectures in
 {T}opology, University of Michigan Press, Ann Arbor, Mich., 1941, 101--141.\LastPageEnding

\end{thebibliography}
\end{document}